\documentclass[12pt,letterpaper, reqno]{amsart}
\makeatletter
\usepackage{fullpage}
\usepackage{textcmds}  
\usepackage{amsmath, amssymb, amsfonts, amstext, verbatim, amsthm, mathrsfs}
\usepackage[mathcal]{eucal}
\usepackage{stmaryrd}
\usepackage{microtype}
\usepackage{graphicx}
\usepackage[all,cmtip,2cell]{xy}
\UseTwocells
\usepackage{pgf,tikz,pgfplots}
\pgfplotsset{compat=1.15}
\usetikzlibrary{arrows}
\usepackage{tikz-cd}
\usepackage{enumerate}
\usepackage{enumitem}
\usepackage{xspace}
\usepackage{cite}
\usepackage{pifont}
\usepackage[modulo]{lineno}
\usepackage[colorlinks=true,linkcolor=blue,citecolor=blue,urlcolor=blue,citebordercolor={0 0 1},urlbordercolor={0 0 1},linkbordercolor={0 0 1}]{hyperref} 
\usepackage[shortalphabetic]{amsrefs} 
\usepackage[nameinlink]{cleveref}

\def\makeCal#1{%
\expandafter\newcommand\csname c#1\endcsname{\mathcal{#1}}}
\def\makeBB#1{%
\expandafter\newcommand\csname b#1\endcsname{\mathbb{#1}}}
\def\makeFrak#1{%
\expandafter\newcommand\csname f#1\endcsname{\mathfrak{#1}}}

\count@=0
\loop
\advance\count@ 1
\edef\y{\@Alph\count@}%
\expandafter\makeCal\y
\expandafter\makeBB\y
\expandafter\makeFrak\y
\ifnum\count@<26
\repeat

\theoremstyle{plain}
\newtheorem{thm}{Theorem}[section]

\newtheorem{conj}[thm]{Conjecture}

\theoremstyle{definition}
\newtheorem{rem}[thm]{Remark}
\newtheorem{defn}[thm]{Definition}

\newtheorem{ex}[thm]{Example}

\newtheorem*{notn*}{Notation}

\DeclareMathOperator{\DCoh}{D^b_{coh}}

\DeclareMathOperator{\Ext}{Ext}

\DeclareMathOperator{\Hom}{Hom}

\DeclareMathOperator{\RHom}{RHom}

\DeclareMathOperator{\Pic}{Pic}

\DeclareMathOperator{\Perf}{Perf}
\newcommand*{\sheafhom}{\mathcal{H} \kern -.5pt om}

\usepackage{pbox}
\usepackage[normalem]{ulem}

\AtBeginDocument{%
   \def\MR#1{}
} 

\def\l@subsection{\@tocline{2}{0pt}{2.5pc}{4pc}{}} 

\makeatletter

\usepackage{babel}
\begin{document}

\title{Non-existence of phantoms on some non-generic blowups of the projective plane}

\author{Lev Borisov and Kimoi Kemboi}

\begin{abstract}
We show that blowups of the projective plane at points lying on a smooth cubic curve do not contain phantoms, provided the points are chosen in very general position on this curve.
\end{abstract}

\maketitle


\section{Introduction}

Let $X$ be a smooth projective variety. A \emph{phantom} on $X$ is a nontrivial admissible subcategory of the derived category of $X$ that is invisible to additive invariants, that is, its Hochschild homology and Grothendieck group vanish. The expectation that phantoms existed on some simply connected surfaces of general type originally came from mirror symmetry \cite{DKK13}. Further evidence came from the construction on the derived category of the classical Godeaux surface of an admissible subcategory with finite but nontrivial Grothendieck group and trivial Hochschild homology \cite{BBS13}, a so-called quasi-phantom.

The first examples of phantoms were constructed on products of surfaces with quasi-phantoms \cite{GO13} and on determinantal Barlow surfaces \cite{BBKS15}. The phantoms on Barlow surfaces are the orthogonal complements of maximal length exceptional collections that are not full. There had been a general expectation that varieties admitting a full exceptional collection should not have phantoms \cite{K14}. This expectation was upended, first by Efimov who showed that any phantom can be embedded in a category admitting a full exceptional collection \cite{E20}, and recently by Krah's elegant example of a phantom on the blowup of $\bP^2$ at ten points in general position \cite{K24}. It is natural to wonder whether phantoms still exist if one takes the blowup at points that are not in general position. In this note, we make the following contribution.

\begin{thm}[\Cref{T:main_theorem}] \label{T:forward_main}
Let $X$ be a smooth complex projective surface with an effective smooth anti-canonical divisor $E$, with the property that the restriction map $\Pic(X)\to \Pic(E)$ is injective. Then $\DCoh(X)$ has no phantoms. 
\end{thm}

\begin{rem}
Typical examples of such surfaces are blowups of $\bP^2$ at a finite set of points on a smooth cubic curve, in very general position on the curve. Specifically, we require that the classes of the blowup points and the pullback of the hyperplane class in $\bP^2$ are linearly independent over $\bQ$. Indeed, in this case, the proper preimage of the cubic curve $E$ is smooth and the above restriction map is injective. 
\end{rem} 

Besides the case of varieties with indecomposable derived categories, it is an interesting and difficult question to determine obstructions to the existence of phantoms on a variety. The only general result in this direction is by Pirozhkov \cite{P23}, who shows the non-existence of phantoms for del Pezzo surfaces. Our proof of \Cref{T:forward_main} relies deeply on the tools developed by Pirozhkov to analyze the support of objects in phantom subcategories, and how it interacts with anti-canonical divisors. This perspective leads us to the following.

\begin{conj}
Let $S$ be a smooth projective surface over an algebraically closed field $k$ of characteristic zero. Assume $H^0(S,\cL) \neq 0$ for $\cL$ either $\omega_S$ or $\omega_S^\vee$. Then $S$ has no phantoms.
\end{conj}

We begin \Cref{S:prelims} with a review of basic properties of admissible subcategories, and the analysis of spherical functors that is the main technical tool in \cite{P23}. We then apply this analysis in \Cref{S:proof} to give a proof of \Cref{T:forward_main}. 

\subsubsection*{Acknowledgments} We thank Johannes Krah and Dmitrii Pirozhkov for helpful comments. The second author also acknowledges the support and hospitality of the Simons Laufer Mathematical Sciences Institute. 

\section{Preliminaries}\label{S:prelims}

\subsection{Admissible categories and phantoms}

Let $\cD$ be a triangulated category.

\begin{defn}
A full triangulated subcategory $\iota : \cB \subset \cD$ is said to be \emph{admissible} if the inclusion functor admits both a right $\iota_R$ and left $\iota_L$ adjoint. 
\end{defn}

When $\cD$ is the derived category of a smooth projective variety, Kuznetsov \cite{K11}*{Thm.~7.1} showed that the right adjoint of $\iota$ is a Fourier-Mukai functor; namely, there exists an object $B_R \in \DCoh(X \times X)$, called a Fourier-Mukai kernel, such that 
\[
    \iota_R (-) = Rq_\ast \left( p^\ast (-) \otimes ^L B_R \right) ,
\]
where $p,q: X\times X \to X$ are respectively the first and second projection functors. This characterization of the right adjoint gives a natural definition of Hochschild homology for an admissible subcategory, see \cite{K11}*{Sect.~7}. We will not require the precise definition of Hochschild homology for an admissible subcategory for reasons that are clarified in \Cref{R:rational_K_0_controls_K_motive}.

\begin{defn}
Let $X$ be a smooth projective variety and let $\cB \subset \DCoh(X)$ be a triangulated subcategory. The Grothendieck group of $\cB$, denoted $K_0(\cB)$, is the abelian group generated by objects of $\cB$ and subject to relations $[F_2] = [F_1] + [F_3]$ for all distinguished triangles $ F_1 \to F_2 \to F_3 \xrightarrow{+   }$ in $\cB$.
\end{defn}

\begin{defn}
Let $X$ be a smooth projective variety. An admissible subcategory $\cB \subset \DCoh(X)$ is said to be \emph{quasi-phantom} if its Hochschild homology  $HH_{\bullet}(\cB)$ is trivial and its Grothendieck group $K_0(\cB)$ is finite. It is a \emph{phantom} if in addition $K_0(\cB) = 0$. 
\end{defn}

\begin{rem} \label{R:rational_K_0_controls_K_motive}
If $X$ is a smooth projective variety over $\bC$, as is the case in \Cref{T:forward_main}, one can drop the condition on vanishing of $HH_\bullet$ in the definition of a phantom. Indeed, over $\bC$, the vanishing of Hochschild homology (and also higher $K$-groups) follows from the vanishing of $K_0$ with rational coefficients \cite{GO13}*{Thm.~5.5}.
\end{rem}

\subsection{Semiorthogonal decompositions} \label{S:SOD}

A semiorthogonal decomposition of $\DCoh(X)$ is a presentation $\DCoh(X) = \langle \cB_1, \ldots, \cB_m \rangle$, where $\cB_i$ are admissible subcategories satisfying
\begin{enumerate}
    \item[(i)] $\RHom_X(b, c) = 0$ for all $b \in \cB_i$ and all $c \in \cB_j$ whenever $i > j$, and
    \item[(ii)] $\DCoh(X)$ is the smallest triangulated subcategory generated by $\cB_1,\ldots, \cB_m$.
\end{enumerate}
The conditions (i) and (ii) imply that if $\DCoh(X) = \langle \cA, \cB \rangle$ is a semiorthogonal decomposition, then every object $F \in \DCoh(X)$ fits in a unique distinguished triangle
\begin{equation*}
    b \to F \to a  \xrightarrow{+   }  \label{eq:triangle_of_sod}
\end{equation*}
with $b \in \cB$ and $a \in \cA$. 

\subsection{Spherical Functors and Spherical Twists}

Let $\cC, \cD$ be triangulated categories and let $F:\cC \to \cD$ be an exact functor admitting right and left adjoints $R,L:\cD \to \cC$. Assume 
\begin{itemize}
    \item [(i)] $\cC$ and $\cD$ admit DG enhancements, and 
    \item [(ii)] the functors $F,R,L$ descend from DG functors between enhancements. 
\end{itemize}
Then there are canonical triangles \cite{AL17} of exact functors

\[
 FR \to \textrm{Id}_\cD \to T  \xrightarrow{+   } \qquad \qquad \textrm{Id}_\cC \to RF \to C  \xrightarrow{+   } .
\]
The functor $T$ is called the \emph{twist}, and $C$ the \emph{cotwist} of $F$.

\begin{defn}[\cites{AL17,A16}]
    The functor $F$ is \emph{spherical} if $C$ is an equivalence and $R\cong CL$. 
\end{defn}

\begin{thm} [\cite{A16}*{Thm.~2.3}]
    If $F:\cC \to \cD$ is a spherical functor, then $T$ is an equivalence. 
\end{thm}

\begin{ex}[Restrictions to divisors]
    Let $X$ be a smooth projective variety over $k$ and let $\iota: D \hookrightarrow X$ be a divisor on $X$. Take $F=L\iota ^\ast : \DCoh(X) \to \Perf(D)$. Then $R= R\iota_\ast$ and $L = \iota_!\cong R\iota_\ast \left(  \omega_\iota[-1] \otimes - \right )$, where $\omega_\iota$ is the relative canonical bundle.
    Pullbacks and pushforwards are Fourier-Mukai functors, so $F,R,L$ descend from DG functors on enhancements \cite{T07}. Note that $RF \cong \cO_D \otimes -$, thus $C = \cO_X(-D) [1]\otimes -$, and $T = \cO_D(-D)[2] \otimes -$ is the spherical twist. It follows that $L\iota^\ast$ is a spherical functor.
\end{ex}

Spherical functors are a generalization of the notion of a spherical object introduced by Seidel-Thomas as a categorification of Dehn twists \cite{ST01}. An object $\cF \in \DCoh(X)$ is \emph{spherical} if 
\[
    \Ext^i_X ( \cF, \cF ) = \left\{
    \begin{array}{ll}
         k & \text{if } i = 0, \dim X \\
         0 & \text{otherwise.}
    \end{array} \right. 
\]
Considering $\DCoh(X)$ as a pre-triangulated dg category, a spherical object $\cF$ defines an autoequivalence $T_\cF: \DCoh(X) \to \DCoh(X)$ of dg-categories,
\[
    T_\cF(a) = \text{Cone} ( \RHom(\cF, a) \otimes \cF \to a),
\]
called the spherical twist around $\cF$. If $\cF \in \DCoh(X)$ is a spherical object, then $F = \cF \otimes - : \DCoh(k) \to \DCoh(X)$ is a spherical functor whose right adjoint is given by $\RHom(\cF, -)$. The cotwist of $F$ is the shift functor $[-n]$ and the twist of $F$ is precisely $T_\cF$ \cite{A16}*{Sect.~2}.

Following the ideas in \cite{P23}, the main technical tool is the autoequivalence, introduced by Addington \cite{A16}, of the derived category of an anti-canonical divisor that is induced by an admissible subcategory $\cB \subset \DCoh(X)$.

\begin{thm}[\cite{P23}*{Thm.3.2, Cor.~3.6}] \label{C:triangle_on_anticanonical} Let $X$ be a smooth projective variety over $k$ and let $j: E \hookrightarrow X$ be an anti-canonical divisor on $X$. Let $\iota: \cB \subset \DCoh(X)$ be an admissible subcategory and let $B_R \in \DCoh(X\times X)$ denote the Fourier-Mukai kernel of the right adjoint of $\iota$. Then the composition $Lj^\ast \iota$ is a spherical functor. In particular, there is an exact triangle in $\DCoh(E\times E)$
\begin{equation}
    B_R\vert_{E\times E} \to \cO_{\Delta_E} \to \cT  \xrightarrow{+   }  \label{eq:triangle_of_kernels}
\end{equation}
such that the Fourier-Mukai functor with kernel $\cT$ induces an autoequivalence $T: \Perf(E) \to \Perf(E)$. Thus any object $\cF \in \Perf(E)$ is part of a distinguished triangle 
\begin{equation}
    Lj^\ast \iota_R (j_\ast \cF) \to \cF \to T(\cF)  \xrightarrow{+   } \label{eq:triangle_on_anticanonical_div}
\end{equation}
in $\Perf(E)$.
\end{thm} 

\begin{proof}
The pullback functor $Lj^\ast: \DCoh(X) \to \Perf(E)$ is spherical with cotwist given by tensoring by $\cO_X(-E)[1]$. Since $E$ is an anti-canonical divisor, 
\cite{A16}*{Prop.~2.1} implies that $Lj^\ast \circ \iota: \cB \to \Perf(E)$ is a spherical functor. So there is a canonical triangle of exact functors that, at the level of Fourier-Mukai kernels, gives the triangle \eqref{eq:triangle_of_kernels}. A direct computation shows that the Fourier-Mukai functor with kernel $B_R\vert_{E\times E}$ coincides with $Lj^\ast \iota_R j_\ast$.
\end{proof}

\section{Main result} \label{S:proof}

\begin{thm}\label{T:main_theorem}
    Let $X$ be a smooth complex projective surface with an effective smooth anti-canonical divisor $E$, with the property that the restriction map $\Pic(X)\to \Pic(E)$ is injective. Then $\DCoh(X)$ has no phantom categories. 
\end{thm}

\begin{proof}[Proof of \Cref{T:main_theorem}]
    
We denote the closed embedding map $E\to X$ by $j$. 
Let $\langle \cA,\cB\rangle$ be a semiorthogonal decomposition of $D^b(X)$ with $\cB$ a phantom category. 

\smallskip
{\bf Step 1.} Let $p$ be a point on $E$ such that no nonzero integer multiple of $[p]$ lies in the image of $\Pic(X)$ under the restriction map. Such points exist and form a dense subset since the complement is a countable subset of $E$. Indeed, $|-K_X|\neq \emptyset$ implies that $X$ is rational, and thus $\Pic(X)$ is finitely generated. We then consider the distinguished triangles
\begin{equation}
    B_p\to j_\ast \cO_p\to A_p \xrightarrow{+   } \label{eq:triangle_1}
\end{equation}
and
\begin{equation}
    Lj^\ast B_p \to \cO_p \to C_p \xrightarrow{+   }, \label{eq:triangle_2}
\end{equation}
 where \eqref{eq:triangle_1} is the triangle induced by the semiorthogonal decomposition (see \Cref{S:SOD}), and \eqref{eq:triangle_2} is the triangle from \Cref{C:triangle_on_anticanonical}. In particular, $C_p$ is a spherical object of $D^b(E)$. Note that since $B_p\in \cB$, the $K$-theory classes of $B_p$ and $Lj^\ast B_p$ are zero. Thus, $C_p$ has the same class as $\cO_p$ and is therefore isomorphic to $\cO_p[2a]$ for some $a\in \bZ$ by classification of spherical objects in $D^b(E)$ \cite{BK06}*{Prop.~4.13}. This implies that $Lj^\ast B_p$ is supported at $p$ (or is zero).

Suppose that $B_p$ is nonzero. By \cite{P23}*{Lem.~6.3}, the support of $B_p$ is a connected closed subset of $X$ that contains $p$. For $q\in E, q\neq p$ we have 
\[
\RHom(B_p, j_\ast\cO_q) = \RHom(Lj^\ast B_p,\cO_q) = 0,
\]
thus ${\rm Supp}(B_p)\cap E =\{p\}$. It follows that ${\rm Supp}(B_p)=\{p\}$ because there are no curves on $X$ that intersect $E$ only at $p$, by our assumption on $p$. However, this contradicts the condition that $\cB$ is a phantom, since phantoms do not contain objects with zero-dimensional support by \cite{P23}*{Lem.~6.16}. We thus conclude that $B_p=0$ which means that $j_\ast\cO_p \in \cA$.
    
\smallskip
{\bf Step 2.}
The first part of the argument of Step 1 shows that $C_p$ is isomorphic to $\cO_p[2a]$ for some integer $a$ for \emph{any} point $p\in E$. Note that this $a$ must be zero for $p$ in very general position, since for such $p$ we proved that $B_p=0$. By \cite{HVdB07}*{Prop.~4.2}, it follows that $a=0$ for all $p$ and the autoequivalence $T$ from \eqref{eq:triangle_on_anticanonical_div} is given by $T \cong \cL \otimes -$. We can apply this autoequivalence to $\cO_E$ to get
\[
Lj^\ast B_{\cO} \to \cO_E\to \cL\stackrel {+}{\to}
\]
where $ B_{\cO} = \iota_R (j_\ast\cO_E)$. Since the $K$-theory class of $Lj^\ast B_{\cO}$ is trivial, we see that $\cL\cong \cO$. Then the morphism of kernels on $E\times E$ that corresponds to ${\rm Id}\to T$ is given by some element $\psi$ of $\Hom_{E\times E}(\cO_\Delta,\cO_\Delta)$. This morphism space is one-dimensional, and the map $\psi$ can not be zero, as this would contradict the conclusion of Step 1 for $p$ in very general position. Consequently, $\psi$  is an isomorphism, thus the functor $Lj^\ast \iota_R j_\ast $ is zero. This implies that $Lj^\ast B_p$ is zero for all $p\in E$. As a result, for any point $q\in E$ we have $\Hom(B_p,j_\ast \cO_q)=0$, which means that the support of $B_p$ is disjoint from $E$, so $B_p$ must be zero. Therefore, for all $p\in E$ we have  $j_\ast \cO_p\in \cA$. 

\smallskip
{\bf Step 3.} We now know that for any $p\in E$ the skyscraper sheaf $j_\ast \cO_p$ lies in $\cA$. If $\cB$ is nontrivial, there exists a closed point $x\in X$ such that in the triangle
\[
B_x\to \cO_x \to A_x \xrightarrow{+   }
\]
the object $B_x$ is nonzero. Since $\Hom(\cB,\cA)=0$, we have 
$\RHom(B_x,j_\ast\cO_p)=0$ for all points $p\in E$. This means that ${\rm Supp}(B_x)$ is disjoint from $E$. By the injectivity of the restriction map, this means that ${\rm Supp}(B_x)$ can not have one-dimensional irreducible components. Since ${\rm Supp}(B_x)$ is connected by \cite{P23}*{Lem.~6.3}, we see that ${\rm Supp}(B_x)=\{x\}$. However, this contradicts the assumption of $\cB$ being a phantom, see \cite{P23}*{Lem.~6.16}.
\end{proof}

\bibliography{bib_phantom}{}
\bibliographystyle{plain}
\end{document}